\documentclass[10pt]{amsart} 
\usepackage{amscd,amsthm,amsfonts,amscd,epsf,amsmath,amssymb,latexsym,multicol}
\usepackage[latin1]{inputenc}
\usepackage{pstricks,pst-node,cite,epsfig}
\usepackage[left=2.50cm, right=2.50cm, top=2.50cm, bottom=2.50cm]{geometry}
\usepackage{graphicx}
\usepackage{xcolor}
\usepackage{mathrsfs}
\usepackage{textcomp}
\usepackage[all,cmtip]{xy}
\usepackage{mathrsfs}
\usepackage{cite}

\newtheorem{thm}{Theorem}[section]
\newtheorem{pro}[thm]{Proposition} 
\newtheorem{rmk}[thm]{Remark} 
\newtheorem{ex}{Example}
\newtheorem{cor}[thm]{Corollary}
\newtheorem{lem}[thm]{Lemma}
\newtheorem{dfn}{Definition}

\title{Cross-connections of the singular transformation semigroup}

\author{P. A. Azeef Muhammed}
\address{School of Mathematics, IISER, Thiruvananthapuram, Kerala-695016, India.\footnote{Present address:\\
Institute of Mathematics and Computer Science,\\ 
Ural Federal University,\\ 
Lenina 51,\\ 
620000 Ekaterinburg, Russia.\\
E-mail address: a.a.parail@urfu.ru}}
\email{$azeefp@gmail.com$}
\author {A. R. Rajan}
\address{Department of Mathematics, University of Kerala, Thiruvananthapuram, Kerala-695581, India.}
\email{$arrunivker@yahoo.com$}
\keywords{Regular semigroup, cross-connections, normal category, transformation semigroup, powerset, partitions}
\subjclass[2010]{20M10, 20M17, 20M50}

\begin{document}

\maketitle

\begin{abstract}
Cross-connection is a construction of regular semigroups using certain categories called normal categories which are abstractions of the partially ordered sets of principal left (right) ideals of a semigroup. We describe the cross-connections in the semigroup $Sing(X)$ of all non-invertible transformations on a set $X$. The categories involved are characterized as the powerset category $\mathscr{P}(X)$ and the category of partitions $\Pi(X)$. We describe these categories and show how a permutation on $X$ gives rise to a cross-connection. Further we prove that every cross-connection between them is induced by a permutation and construct the regular semigroups that arise from the cross-connections. We show that each of the cross-connection semigroups arising this way is isomorphic to $Sing(X)$. We also describe the right reductive subsemigroups of $Sing(X)$ with the category of principal left ideals isomorphic to $\mathscr{P}(X)$. This study sheds light into the more general theory of cross-connections and also provides an alternate way of studying the structure of $Sing(X)$.
\end{abstract}

\section{Introduction}
In the development of the structure theory of regular semigroups, there are mainly two approaches. The first approach inspired by the work of W. D. Munn \cite{munn} for inverse semigroups uses the set of idempotents $E$ of a given regular semigroup to reconstruct the semigroup as a full subsemigroup of the semigroup of principal ideal isomorphisms of $E$. In this context, K. S. S. Nambooripad \cite{mem} abstractly characterized the set of idempotents of a (regular) semigroup as a \emph{(regular) biordered set}. It was later proved by D. Easdown \cite{eas} that given a biordered set $E$, there exists a semigroup $S$ such that the biordered set of $S$ is isomorphic to $E$.

The second approach initiated by T. E. Hall \cite{hall} uses the ideals of a given regular semigroup to analyse its structure. P. A. Grillet \cite{gril,gril1,gril2} refined Hall's theory to abstractly characterize the set of ideals as \emph{regular partially ordered sets} and constructed the fundamental image of the regular semigroup as a cross-connection semigroup. Again Nambooripad \cite{cross} generalized the idea to arbitrary regular semigroups by characterizing the set of ideals with appropriate morphisms as \emph{normal categories}. A \emph{cross-connection} between two normal categories determines a cross-connection semigroup and conversely every regular semigroup is isomorphic to a cross-connection semigroup for a suitable cross-connection.

Nambooripad's cross-connections are thus much more general than Grillet's and allows one to coordinatize arbitrary regular semigroups, and not only fundamental ones. As usual there is a price to pay: Nambooripad's theory involves rather a complicated system of notions and is more difficult to work with. Because of this and also of the circumstance that Nambooripad's treatise \cite{cross} was not easily available, the ideas of \cite{cross} have not found many applications so far.

In this paper, we present such an application describing the cross-connections of the singular part $Sing(X)$ of the full transformation semigroup on a set $X$, i.e., of the semigroup of all non-invertible  transformations of $X$. The purpose is twofold. On the one hand, we demonstrate how Nambooripad's theory works in a very concrete yet non-trivial situation. As the reader will see, all the abstract notions of the theory get rather transparent meanings in this case. On the other hand, we believe that our approach sheds some new light on the structure of $Sing(X)$. With each transformation $a \in Sing(X)$, its image Im $a$ and its kernel $\pi_a$ are connected, and naturally, the set $\mathscr{P}(X)$ of all non-empty proper subsets of $X$ and the set $\Pi(X)$ of all non-identity partitions of $X$ play an important role in understanding the structure of $Sing(X)$. (For instance, it is well known that $\mathscr{P}(X)$ and $\Pi(X)$ characterize Green's relations on $Sing(X)$ \cite{clif}.) Roughly speaking, our results show what structure should be imposed on $\mathscr{P}(X)$ and $\Pi(X)$ in order to make $Sing(X)$ become reconstructable from them. Clearly, both $\mathscr{P}(X)$ and $\Pi(X)$ come equipped with their natural order but one needs more ingredients, and we show that such additional ingredients can be comfortably expressed in the language of category theory.

The singular transformation semigroup $Sing(X)$ is a regular subsemigroup of the full transformation semigroup $\mathscr{T}_X$. Since every semigroup can be realized as a transformation semigroup, $\mathscr{T}_X$ and its subsemigroups have been studied extensively \cite{igd,howie1978,mstx,cfts}. $Sing(X)$ being a regular semigroup with a relatively nice structure makes a good candidate to describe some non-trivial cross-connections because it is not a monoid. (One can see that the cross-connection structure degenerates in the case of a regular monoid.) In the process, we identify the normal categories involved in the construction, their normal duals, some intermediary semigroups that arise, the cross-connections between the normal categories, and also the semigroups that arise from the cross-connections. The first author has done a similar analysis on the linear transformation semigroup \cite{tlx} in which, further, the construction yields an interesting structure theorem for the regular part of the so-called \emph{variants} of the semigroup, in the sense of I. Dolinka and J. East \cite{igd}.

The structure of the paper is as follows. In Section \ref{sccxn}, we discuss the relevant notions and definitions in the general theory of cross-connections. Although the \emph{abstract nonsense} is slightly bulky, it is necessary if one wants to place our results into a proper context and relate them to the general situation. A reader more interested in the application of the theory to $Sing(X)$ may skip this section since we provide a sort of restriction of the general theory to $Sing(X)$ in Section \ref{secpxc}. The normal categories arising from the singular transformation semigroup $Sing(X)$ are characterized as the powerset category and the category of partitions. We begin Section \ref{seccxn} by characterizing the \emph{normal duals}. We observe that the normal dual of the powerset category is isomorphic to the category of partitions, and vice versa, the normal dual of the partition category is isomorphic to the powerset category. We study some cross-connections between the powerset category and the partition category and describe how each permutation on the set $X$ gives rise to a cross-connection. We prove that \emph{every} cross-connection semigroup that arises from these categories is in fact induced by a permutation on $X$. In Section \ref{secrrs}, we conclude the paper by constructing the right reductive subsemigroups of $Sing(X)$ as cross-connection semigroups. 

\section{The theory of cross-connections}\label{sccxn}
In this section, we briefly describe the basic concepts and definitions in the general theory of cross-connections. It is based on \cite{cross}; and a more concise but a less detailed discussion on cross-connection theory is available at \cite{kvn}. We assume familiarity with the elementary ideas in semigroup theory and the basic definitions in category theory like functors, natural transformations etc. A reader unfamiliar with these notions may refer to \cite{clif,grillet,howie} for semigroup theory; and \cite{bucur,mac} for basic category theory. In this paper, all the functions are written in the order of their composition, i.e., from left to right (with a minor exception which shall be seen later). All the categories considered are \emph{small categories} and for a category $\mathcal C$, the set of objects of $\mathcal C$ is denoted by $v\mathcal C$ and the set of morphisms by $\mathcal C$ itself. Thus the set of all morphisms between objects $c,d \in v\mathcal{C}$ is denoted by $\mathcal{C}(c,d)$. Observe that since categories are treated as algebraic objects, to compare two categories, we will require the stronger notion of category isomorphism (category equivalence will not suffice).

Let $\mathcal{C}$ be a category and $\mathcal{P}$ be a subcategory of $\mathcal{C}$ such that $\mathcal{P}$ is a strict preorder with $v\mathcal{P} = v\mathcal{C}$. Then $(\mathcal{C} ,\mathcal{P})$ is called a \emph{category with subobjects}\index{category!with subobjects} if, first, every $f\in \mathcal{P}$ is a monomorphism in $\mathcal{C}$, second, and if $f=hg$ for $f,g\in \mathcal{P}$ and $h \in \mathcal{C}$, then $h\in \mathcal{P}$.

In a category $(\mathcal{C} ,\mathcal{P})$ with subobjects, if $f\colon c \to d $ is a morphism in $ \mathcal{P}$, then $f$ is said to be an \emph{inclusion}\index{inclusion}. We denote this inclusion by $j(c,d)$. 
An inclusion $j(c,d)$ \emph{splits} if there exists $e\in \mathcal{C}$ such that $j(c,d)e =1_c$. Then $e$ is called a \emph{retraction}. A morphism $u\colon c \to d$ is said to be an \emph{isomorphism}\index{isomorphism} if there exists a morphism $u^{-1}\colon d \to c$ such that $uu^{-1} = 1_c$ and $u^{-1}u = 1_d$. 
\begin{dfn}
A \emph{normal factorization}\index{factorization!normal} of a morphism $f \in \mathcal{C}(c,d)$ is a
factorization of the form $f=euj$ where $e\colon c\to c'$ is a retraction,
$u\colon c'\to d'$ is an isomorphism and $j=j(d',d)$ an inclusion where $c',d' \in v\mathcal{C}$ with $ c' \subseteq c$, $ d' \subseteq d$.\\
The morphism $f^\circ = eu$ is known as the \emph{epimorphic component}\index{epimorphic component!of a morphism} of the morphism $f$.
\end{dfn}
The below diagram illustrates a normal factorization $f=euj$ of $f\in\mathcal{C}(c,d)$.
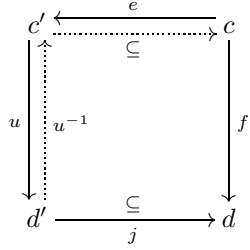
\begin{equation*}\label{normalfact}
\xymatrixcolsep{5pc}\xymatrixrowsep{5pc}\xymatrix
{
 c' \ar@{.>}@<-3pt>[r]_{\subseteq} \ar@<-3pt>[d]_{u}   & c \ar@<-3pt>[l]_{e}\ar[d]^{f} \\       
 d' \ar@{.>}@<-3pt>[u]_{u^{-1}} \ar[r]^{\subseteq}_{j} & d 
}
\end{equation*}
A \emph{normal cone} in $\mathcal C$ is essentially a {nice} collection of morphisms in $\mathcal C$ associated with a distinguished vertex $d \in v\mathcal C$. 
\begin{dfn}\label{dfn1}
Let $\mathcal{C}$ be a category with subobjects and $d\in v\mathcal{C}$. A map $\gamma\colon v\mathcal{C}\to\mathcal{C}$ is called a \emph{cone}\index{cone} from the base $v\mathcal{C}$ to the vertex $d$ if $\gamma(c)\in \mathcal{C}(c,d)$ for all $c\in v\mathcal{C}$ and whenever $c\subseteq c'$ then $j(c,c')\gamma(c') = \gamma(c)$. A cone $\gamma$ with vertex $d$ is said to be \emph{normal}\index{cone!normal} if there exists $c\in v\mathcal{C}$ such that $\gamma(c)\colon c\to d$ is an isomorphism.
\end{dfn}
Given the cone $\gamma$, we denote by $c_{\gamma}$ the \emph{vertex}\index{vertex!of a cone} of
$\gamma$ and for each $c\in v\mathcal{C}$, the morphism $\gamma(c)\colon  c \to c_{\gamma}$
is called the \emph{component}\index{component!of a cone} of $\gamma$ at $c$. A normal cone $\gamma$ with vertex $d$ can be represented as follows.
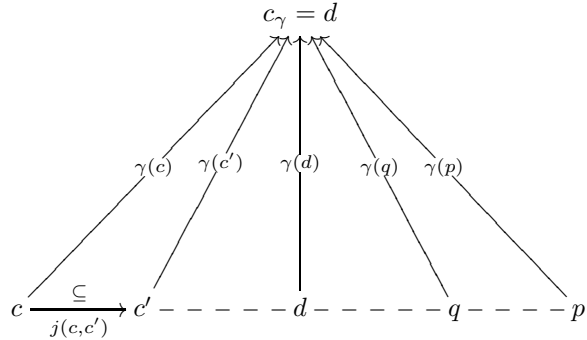
\begin{equation*}\label{conediag}
\xymatrixrowsep{8pc}\xymatrixcolsep{3pc}\xymatrix
{
&&c_\gamma=d\\
c\ar[r]^\subseteq_{j(c,c')} \ar[rru]|-{\gamma(c)} &
c' \ar[ru]|-{\gamma(c')} &
d \ar@{--}[l] \ar[u]|-{\gamma(d)} &
q \ar@{--}[l] \ar[lu]|-{\gamma(q)} &
p \ar@{--}[l] \ar[llu]|-{\gamma(p)}
}
\end{equation*}
We define the \emph{M-set}\index{M-set} of a cone $\gamma$ as 
\begin{equation}
M\gamma = \{ c \in \mathcal{C}:\gamma(c)\text{ is an isomorphism} \}.
\end{equation} 
\begin{dfn}
A small category $\mathcal{C}$ with subobjects is called a \emph{normal category}\index{category!normal} if the following holds:
\begin{enumerate}
\item Any morphism in $\mathcal{C}$ has a normal factorization. 
\item Every inclusion in $\mathcal{C}$ splits.
\item For each $c \in v\mathcal{C} $ there is a normal cone $\gamma$ with vertex $c$ and $\gamma (c) = 1_c$.
\end{enumerate}
\end{dfn}
Observe that given a normal cone $\gamma$ and an epimorphism $f\colon c_\gamma \to d$, the map $\gamma*f \colon  a \mapsto\gamma(a)f$ from $v\mathcal{C}$ to $\mathcal{C}$ is a normal cone with vertex $d$. Hence, given two normal cones $\gamma$ and $\sigma $, we can compose them as follows.
\begin{equation} \label{eqnsg1}
\gamma \cdot \sigma = \gamma*(\sigma(c_\gamma))^\circ 
\end{equation} 
where $(\sigma(c_\gamma))^\circ$ is the epimorphic\index{epimorphic component!of a morphism} part of the morphism $\sigma(c_\gamma)$.

All the \emph{normal cones} in a normal category with this special binary composition form a regular semigroup $T\mathcal C$ known as the \emph{semigroup of normal cones} in $\mathcal C$. 

To describe cross-connections, Grillet \cite{gril} used the set of all \emph{normal} equivalence relations on a \emph{regular} partially ordered set. To extend this idea to normal categories, Nambooripad proposed the notion of a \emph{normal dual}. The normal dual $N^\ast\mathcal C$ of a normal category $\mathcal{C}$ is a full subcategory of the category $\mathcal{C}^\ast$ of all functors from $\mathcal C$ to $\bf{Set}$ \cite{mac}. The objects of $N^\ast\mathcal C$ are functors called $H$-functors and the morphisms are natural transformations between them. 

For each $\gamma \in T\mathcal{C}$, define the \emph{H-functor}\index{functor!H-} $H({\gamma};-)\colon  \mathcal{C}\to \mathbf{Set}$ on the objects and morphisms of $\mathcal{C}$ as follows. For each $c\in v\mathcal{C}$ and for each $g\in \mathcal{C}(c,d)$, we let
\begin{subequations} \label{eqnH11}
\begin{align}
H({\gamma};{c})&= \{\gamma\ast f^\circ : f \in \mathcal{C}(c_{\gamma},c)\} \text{ and }\\
H({\gamma};{g}) :H({\gamma};{c}) &\to H({\gamma};{d}) \text{ given by }\gamma\ast f^\circ \mapsto \gamma\ast (fg)^\circ
\end{align}
\end{subequations}
It can be shown that if $H({\gamma};-) = H({\gamma'};-)$, then the $M$-sets of the normal cones $\gamma$ and $\gamma'$ coincide; and hence we define the $M$-set of an $H$-functor as $MH(\gamma;-) = M\gamma$.
\begin{dfn}
If $\mathcal{C}$ is a normal category, the \emph{normal dual} $N^\ast \mathcal{C}$ is a category with
\begin{equation} \label{eqnH1}
v N^\ast \mathcal{C} = \{ H(\epsilon;-) \: : \: \epsilon \in E(T\mathcal{C}) \}.
\end{equation} 
A morphism in $N^*\mathcal{C}$ between two $H$-functors $H(\epsilon;-)$ and $H(\epsilon';-)$ is a natural transformation $\sigma$ as described in the following commutative diagram.
\begin{equation*}\label{Hfunct}
\xymatrixcolsep{4pc}\xymatrixrowsep{3pc}\xymatrix
{
 H(\epsilon;c) \ar[r]^{\sigma(c)} \ar[d]_{H(\epsilon;g)}  
 & H(\epsilon';c) \ar[d]^{H(\epsilon';g)} \\       
 H(\epsilon;d) \ar[r]^{\sigma(d)} & H(\epsilon';d) 
}
\end{equation*}
\end{dfn}
\begin{thm} \cite{cross} \label{thm1}
Let $\mathcal{C}$ be a normal category. $N^\ast \mathcal{C}$\index{normal dual} is a normal category such that for every morphism $\sigma \colon  H(\epsilon;-) \to H(\epsilon ';-)$ in $N^\ast \mathcal{C} $, there is a unique $\hat{\sigma} \colon  c_{\epsilon '} \to c_\epsilon$ in $\mathcal{C}$ where the component of the \emph{natural transformation}\index{natural transformation} $\sigma$ at $c$ is the map given by:
\begin{equation} \label{eqnH2}
\sigma(c) \colon  \epsilon \ast f^\circ \longmapsto \epsilon ' \ast (\hat{\sigma} f)^\circ. 
\end{equation}
\end{thm}
\begin{dfn}\label{ideal}
Let $\mathcal{C}$ be a small category with subobjects. An \emph{ideal}\index{ideal!of a category} $\langle c \rangle$ of $\mathcal{C}$ is the full subcategory of $\mathcal{C}$ whose objects are subobjects of $c$ in $\mathcal{C}$. It is called the principal ideal generated by $c$.
\end{dfn}
\begin{dfn} \label{lociso}
Let $\mathcal{C}$ and $\mathcal{D}$ be normal categories. A functor $F\colon  \mathcal{C} \to \mathcal{D}$ is said to be a \emph{local isomorphism}\index{local isomorphism} if $F$ is inclusion preserving, fully faithful and for each $c \in v\mathcal{C}$, $F_{|\langle c \rangle}$ is an isomorphism of the ideal $\langle c \rangle$ onto $\langle F(c) \rangle$.
\end{dfn}
\begin{dfn} \label{cxn}
Let $\mathcal{C}$ and $\mathcal{D}$ be normal categories. A \emph{cross-connection}\index{cross-connection} from $\mathcal{D}$ to $\mathcal{C}$ is a triplet $(\mathcal{D},\mathcal{C};{\Gamma})$ where $\Gamma\colon  \mathcal{D} \to N^\ast\mathcal{C}$ is a local isomorphism such that for every $c \in v\mathcal{C}$, there is some $d \in v\mathcal{D}$ such that $c \in M\Gamma(d)$.
\end{dfn}
\begin{rmk}\label{rmkdual}
Given a cross-connection $(\mathcal{D},\mathcal{C};{\Gamma})$, there always exists a unique \emph{dual cross-connection} $\Delta$ from $\mathcal{C}$ to $N^\ast\mathcal{D}$. Further every result of cross-connection gives a corresponding result for the dual cross-connection.   
\end{rmk}
Given the cross-connection $\Gamma$ with the dual cross-connection $\Delta$, by category isomorphisms \cite{mac}, we have two associated bifunctors $\Gamma(-,-)\colon  \mathcal{C}\times\mathcal{D} \to \bf{Set}$ and $\Delta(-,-)\colon  \mathcal{C}\times\mathcal{D} \to \bf{Set}$ such that there is a natural isomorphism $\chi_\Gamma$ between them. Using the bifunctors, we obtain the following intermediary regular semigroups which are subsemigroups of $T\mathcal C$ and $T\mathcal D$ respectively.
\begin{subequations}
\begin{align}
U\Gamma = & \bigcup\: \{  \Gamma(c,d) : (c,d) \in v\mathcal{C} \times v\mathcal{D} \} \\
U\Delta = & \bigcup\: \{  \Delta(c,d) : (c,d) \in v\mathcal{C} \times v\mathcal{D} \}
\end{align}
\end{subequations}
The natural isomorphism $\chi_\Gamma$ is called the \emph{duality} associated with the cross-connection. Using $\chi_\Gamma$ we can get a \emph{linking} of some normal cones in $U\Gamma$ with those in $U\Delta$. Given a cross-connection $\Gamma$ with the dual $\Delta$, a cone $\gamma \in U\Gamma$ is said to be \emph{linked}\index{cone!linked} to $\delta \in U\Delta$ if there is a $(c,d) \in v\mathcal{C} \times v\mathcal{D}$ such that $\gamma \in \Gamma(c,d)$ and $ \delta = \chi_\Gamma(c,d)(\gamma)$. The pairs of linked cones $(\gamma,\delta)$ will form a regular semigroup which is called the \emph{cross-connection semigroup} $\tilde{S}\Gamma$ determined by $\Gamma$.
$$ \tilde{S}\Gamma = \:\{\: (\gamma,\delta) \in U\Gamma\times U\Delta : (\gamma,\delta) \text{ is linked }\:\} $$ 
For $(\gamma,\delta),( \gamma' , \delta') \in \tilde{S}\Gamma$, the binary operation is defined by 
$$ (\gamma , \delta) \circ ( \gamma' , \delta') = (\gamma . \gamma' , \delta' . \delta)    $$

Let $S$ be an arbitrary regular semigroup. The objects of the category $\mathcal{L}$ of principal left ideals are $Se$ for $e\in E(S)$, and the morphisms are partial right translations $\rho_u: u \in eSf$. Dually, the objects of the category $\mathcal{R}$ of principal right ideals are $eS$, and the morphisms are partial left translations $\lambda_w : w \in fSe$. They have an obvious \emph{choice of subobjects}, namely the preorder induced by set inclusions. It can be seen that they form normal categories with this choice of subobjects. Further if we do the cross-connection construction starting with the categories $\mathcal{L}$ and $\mathcal{R}$, then $\tilde{S}\Gamma$ will be isomorphic to $S$. This will give a faithful representation of the semigroup $S$ as a subsemigroup of $T\mathcal L \times (T\mathcal R)^\text{op}$, being a sub-direct product.

\section{Powerset category and the category of partitions}\label{secpxc}
Let $X$ be a non-empty set and in order to avoid trivialities, assume that $X$ has at least two elements. Given the set $X$, we can construct a very simple category $\mathscr{P}(X)$ from $X$ with the object set as the set of all proper subsets of $X$. Given any two subsets $A$, $B$; a morphism from $A$ to $B$ is a function $f$ from $A$ to $B$. We will call $\mathscr{P}(X)$ the \emph{powerset category} and it has a natural \emph{choice of subobjects}: the one provided by set inclusions. Here, the composition of two morphisms is function composition and inclusion morphisms are inclusion functions, i.e, we have an inclusion $j = j(A,B) \colon  A \to B$ if $A\subseteq B$. Given an inclusion $j(A,B)$, it is easy to see that there exists a function $q\colon  B \to A$ such that $j(A,B)q = 1_{A}$. Thus every inclusion splits in $\mathscr{P}(X)$ and $q$ is a retraction.

Now given any function $f$ from $A$ to $B$, it has a factorization of the form $f = quj$ where $q\colon A\to A'$ is a retraction, $u\colon A'\to B'$ is an isomorphism (i.e., a bijection) and $j=j(B',B)$ is an inclusion where $B' =$ Im $f$ and $A'$ is a cross-section of the partition $\pi_f$ of $A$ determined by $f$. Given a partition on a set $A$, a subset $A'$ of $A$ is called a \emph{cross-section}\index{cross-section} provided that $A'$ contains exactly one representative from each equivalence class. Note that $\pi_f \:=\: \{ (b)f^{-1} : b \in B'\}$. This is the \emph{normal factorization} of $f$ and $qu$ is the \emph{epimorphic component} $f^\circ$.

Given any  $D \subseteq X$, we associate with it a function $\sigma\colon v\mathscr{P}(X) \to \mathscr{P}(X)$ having the following properties. For each subset $A$ of $X$, there exists a map $\sigma(A)\colon  A \to D$, and whenever $A \subseteq B$, $j(A,B)\sigma(B) =\sigma(A)$. Further for some subset $C$ of $X$, the map $\sigma(C)\colon C \to D$ is a bijection. This collection of morphisms $\{\sigma(A) : A\in v\mathscr{P}(X)\}$ constitute a \emph{normal cone} $\sigma$ with vertex $D$ in the category $\mathscr{P}(X)$. In addition if $\sigma(D) =1_D$, then $\sigma$ will be called an \emph{idempotent} normal cone.

Let $u\colon  X \to D $ be a transformation such that $(x)u = x $ for every $x \in D$. Then for any $A \subseteq X$, define 
$$\sigma(A) = u_{|A} \colon  A \to D.$$
Then $\sigma$ is an idempotent normal cone since $\sigma(D) = 1_D$. So $\mathscr{P}(X)$ is a small category with subobjects such that each morphism in $\mathscr{P}(X)$ has a normal factorization and each $D \in v\mathscr{P}(X)$ has an associated idempotent normal cone; and hence $\mathscr{P}(X)$ is a normal category as defined in Section \ref{sccxn}.

Let $\gamma$ be a normal cone with vertex $C$ and $f\colon  C\to D$ an onto map in $\mathscr{P}(X)$. Then we can see that $\gamma*f\colon  A\mapsto \gamma(A) f$ gives a normal cone in $\mathscr{P}(X)$ with vertex $D$.

Now suppose $\gamma$, $\delta$ are two normal cones in $\mathscr{P}(X)$ with vertices $C$ and $D$ respectively, we can compose them as follows. 
\begin{equation} \label{eqnsg}
\gamma \cdot \delta = \gamma*(\delta(C))^\circ
\end{equation} 
where $(\delta(C))^\circ$ is the epimorphic component of the morphism $\delta(C)$. Then $\gamma \cdot \delta$ is a normal cone with vertex $D$. The set of all normal cones in $\mathscr{P}(X)$ under the binary operation defined in (\ref{eqnsg}) forms a regular semigroup $T\mathscr{P}(X)$ called the \emph{semigroup of normal cones} in $\mathscr{P}(X)$.

It can be shown that every normal cone $\sigma$ in $\mathscr{P}(X)$ with vertex $A$ defines a transformation $a\colon  X \to A$ as follows.
\begin{equation}\label{cone}
(x) a = (x)\sigma(\{x\}) \text{ for all } x \in X
\end{equation} 
where $\sigma(\{x\})$ is the component of $\sigma$ at the singleton $\{x\} \in v \mathscr{P}(X)$; and conversely every transformation $a\colon X \to A$ determines a normal cone in $\mathscr{P}(X)$ \cite{ltx}.

If we denote this normal cone generated by the transformation $a$ as $\rho^a$; then we can show that $\phi\colon  Sing(X) \to T\mathscr{P}(X) : a \mapsto \rho^a$ is a semigroup homomorphism. Such a homomorphism always exists between a regular semigroup $S$ and $T\mathcal L$ where $\mathcal L$ is the normal category associated with $S$ \cite[Theorem 3.16]{cross}. Although in general, $\phi$ need not be injective or onto, in the case of $Sing(X)$, we can show that $\phi$ is indeed a bijection. Summarising, we have the following.
\begin{thm}\cite{ltx}\label{thmpx}
$\mathscr{P}(X)$ is a normal category and $T\mathscr{P}(X)$ is isomorphic to $Sing(X)$.
\end{thm}  
\begin{rmk}
Observe that $T\mathscr{P}(X)$ is a representation of the semigroup $Sing(X)$; wherein an element in $Sing(X)$ is represented as a normal cone in $\mathscr{P}(X)$, which is nothing but a generalization of \emph{normal mapping} of Grillet \cite{gril}.
\end{rmk}
Now using the category $\mathscr{P}(X)$, we construct the \emph{normal dual} $N^*\mathscr{P}(X)$ of $\mathscr{P}(X)$. The objects of $N^*\mathscr{P}(X)$ are set-valued functors called $H$-functors. For an idempotent transformation $e \in Sing(X)$, define a $H$-functor $ H(e;-) \colon  \mathscr{P}(X) \to \mathbf{Set}$ as follows. For each $A \in v\mathscr{P}(X)$ and for each $g\colon A\to B$ in $\mathscr{P}(X)$,
\begin{subequations} \label{eqnH}
\begin{align}
H(e;A)= &\{ef : \: f:\text{Im }e\to A \} \text{ and }\\
H(e;g) :H(e;A) &\to H(e;B) \text{ given by }ef \mapsto efg.
\end{align}
\end{subequations}
\begin{lem}\cite{pix}\label{lemH}
Let $e \in  Sing(X) $ and $A\subseteq X$. Then
$$ H (e ; {A}) = \{  a \in Sing(X) \: : \: \pi_a \supseteq \pi_e \text{ and Im }a\subseteq A  \}$$
and if $g : {A}\to{B} $ then $ H (e ; g ) \colon  H(e;{A})  \to  H(e;{B})  $ is given by $a \mapsto ag$.
\end{lem}
Thus the $H$-functor $ H (e ;-)$ is completely determined by the partition of $e$. The morphisms in $N^\ast \mathscr{P}(X)$ which are natural transformations between the functors are characterized as follows.
\begin{lem}\cite{pix}\label{lemH2} 
Let $\sigma\colon  H(e;-) \to H(f;-)$ be a morphism in $N^\ast \mathscr{P}(X)$. Then there exists a unique $w \in f(Sing(X))e$ such that the component of the natural transformation $ \sigma({C}) \colon  H(e;C) \to H(f;C)$ is given by $ a \mapsto wa$ where $a \in H(e; {C})$ so that $\pi_a \supseteq \pi_e \text{ and Im }a\subseteq C$.
\end{lem}
The above discussion regarding the normal dual of $\mathscr{P}(X)$ inspires us to define a category of partitions of $X$ as follows. A \emph{partition}\index{partition} $\pi$ of $X$ is a family of subsets $A_i$ of $X$ such that $\cup A_i \: = \: X$ and $A_i\cap A_j = \phi$ for $i \ne j$. A partition is said to be non-identity if at least one $A_i$ has more than one element. Given a non-identity partition $\pi$ of $X$, we denote by $\bar{\pi}$ the set of all functions from $\pi$  to $X$.

Let $\pi_1$, $\pi_2$ be partitions of $X$ and let ${\eta}$ be a function from $\pi_2$ to $\pi_1$. We define $\eta^\ast \colon  \bar{\pi}_1 \to \bar{\pi}_2$ by $(\alpha)\eta^\ast = {\eta}\alpha$ for every $\alpha \in \bar{\pi}_1$. 

Now the \emph{category of partitions}\index{category!of partitions} $\Pi(X)$ of the set $X$ has vertex set 
$$v\Pi(X) = \{ \: \bar{\pi} \: :\:\pi \text{ is a non-identity partition of }X \}$$
and a morphism in $\Pi(X)$ from $\bar{\pi}_1 \to \bar{\pi}_2$ is given by $\eta^\ast$ as defined above where $\eta\colon  \pi_2 \to \pi_1$. Observe that $\eta$ from $\pi_2$ to $\pi_1$ gives a morphism from $\bar{\pi}_1$ to $\bar{\pi}_2$; and it is precisely for this reason, the category of partitions is defined as above.

If $\pi_1 = \{ A_i : i \in I\}$ and $\pi_2 = \{ B_j : j \in J \}$, define a partial order on $\Pi(X)$ as 
\begin{gather*}
\bar{\pi}_1 \leq \bar{\pi}_2\text{ if and only if } {\pi_2} \leq {\pi_1} \text{; where }\\
{\pi_2} \leq {\pi_1}\text{ if and only if for each }j, B_j \subseteq A_i\text{ for some }i.
\end{gather*}
If $ {\pi_2} \leq {\pi_1}$, then the map $\nu\colon  B_j \mapsto A_i$ is the inclusion map from $\pi_2 \to \pi_1$ and $\nu^\ast \colon  \bar{\pi}_1 \to \bar{\pi}_2$ is a morphism in $\Pi(X)$. We consider $\nu^\ast$ as the inclusion\index{inclusion} morphism from $\bar{\pi}_1$ to $\bar{\pi}_2$.

If $\nu^\ast = j(\bar{\pi}_1,\bar{\pi}_2) \colon  \bar{\pi}_1 \to \bar{\pi}_2$ is the inclusion morphism, then there exists a retraction\index{retraction} $\zeta^\ast \colon  \bar{\pi}_2 \to \bar{\pi}_1$, i.e., $\nu^\ast\zeta^\ast = 1_{\bar{\pi}_1}$. Here, $\zeta\colon  \pi_1 \to \pi_2$ is defined by $(A_i)\zeta = B_j$ where $B_j \in \pi_2$ is chosen such that $ B_j \subseteq A_i $. Thus the inclusion $\nu^\ast$ splits.

Now if we define $\mathcal{P}$ as the subcategory of $\Pi(X)$ with $v\mathcal{P} = v\Pi(X)$ and morphisms as inclusion morphisms $\nu^\ast$, then we can verify that $(\Pi(X),\mathcal{P})$ forms a category with subobjects. The next theorem gives a normal factorization of a morphism in $\Pi(X)$.

\begin{thm}
Let $\eta^\ast$ be a morphism from $\bar{\pi}_1$ to $\bar{\pi}_2$ in $\Pi(X)$ with $\eta \colon  \pi_2 \to \pi_1$. Then there exist partitions $\pi_{\sigma}$ and $\pi_{\gamma}$; and $\nu\colon  \pi_2 \to \pi_{\sigma}$, $u\colon  \pi_{\sigma} \to \pi_{\gamma}$ and $\zeta\colon \pi_{\gamma} \to \pi_1$ such that $\eta^\ast = {\zeta^\ast}{u^\ast}{\nu^\ast}$ such that $\zeta^\ast$ is a retraction, $u^\ast$ is an isomorphism and $\nu^\ast$ is an inclusion in $\Pi(X)$.
\end{thm}
\begin{proof}
For each partition $\pi = \{C_k : k \in K \}$ and $x \in X$, we denote by $[x]= [x]_\pi$ the set $C_k \in \pi$ such that $x \in C_k$. That is $[x]_\pi$ is the equivalence class of $x$ in $\pi$, considering $\pi$ as an equivalence relation. Let $\sigma = \sigma_\eta$ be the equivalence relation on $X$ defined as follows.
$$ \sigma =\{ (x,y) \: :\: [x]_{\pi_2}\eta = [y]_{\pi_2}\eta \} $$
Then as equivalence relations $\pi_2 \subseteq \sigma$. Let $\nu \colon  \pi_2 \to \sigma$ be the inclusion given by $[x]_{\pi_2} \mapsto [x]_\sigma$. Since $\nu\colon  \pi_2 \to \sigma$ is the inclusion map, we see that $\nu^\ast \colon  \bar{\sigma} \to \bar{\pi}_2$ is the inclusion morphism.

Let $\gamma = \gamma_\eta$ be the partition of $X$ defined as follows. Let $\pi_1 = \{ A_i : i \in I\}$. Then Im $\eta =\{ A_i : i \in I'\}$ for some $I'\subseteq I$. Then fix an element $1 \in I'$ and define $\gamma = \{ B \cup A_1\} \cup \{ A_i : i \in I'\setminus\{1\} \}$ where $B = \cup\{ A_j : A_j \notin \text{Im }\eta \} $. Then clearly $\gamma$ is a partition of $X$ and $\pi_1 \subseteq \gamma$.

Let $\zeta \colon  \gamma \to \pi_1$ be defined as follows.
\begin{equation*}
\begin{split}
B\cup A_1 &\mapsto A_1\\
A_j & \mapsto A_j \text{ for } j \in I'\setminus\{1\}
\end{split}
\end{equation*}
Then we see that $\zeta^\ast$ is a retraction from $\bar{\pi}_1$ to $\bar{\gamma}$. Now define $u\colon  \sigma \to \gamma$ as follows.
$$([x]_\sigma) u=
\begin{cases}
[x]_{\pi_2}\eta & \text{ if } [x]\eta \neq A_1\\
B\cup A_1 & \text{ if } [x]\eta = A_1
\end{cases}$$
Clearly $u$ is well-defined and is a bijection. Since $u$ is a bijection, $u^\ast \colon  \bar{\gamma} \to \bar{\sigma}$ is also a bijection and so is an isomorphism in $\Pi(X)$.

To show that $\eta^\ast = \zeta^\ast u^\ast \nu^\ast$, for any $[x]_{\pi_2} \in \pi_2$,
\vspace*{.3cm} 
\begin{align*}
([x]_{\pi_2})\nu u \zeta\: = \: ([x]_\sigma)u \zeta \:&=\: 
\begin{cases}
([x]_{\pi_2}\eta)\zeta & \text{ if } [x]\eta \neq A_1\\
(B\cup A_1)\zeta & \text{ if } [x]\eta = A_1
\end{cases}\\  
&=\:
\begin{cases}
[x]_{\pi_2}\eta & \text{ if } [x]\eta \neq A_1\\
A_1 & \text{ if } [x]\eta = A_1
\end{cases}\\  
&=\: [x]_{\pi_2}\eta.
\end{align*}
Hence, for any $\alpha \in \bar{\pi}_1$, 
$$(\alpha)\eta^\ast = (\alpha)(\nu u \zeta)^\ast = \nu u \zeta\alpha = (\alpha){\zeta}^\ast{u}^\ast{\nu}^\ast.$$
Thus $\eta^\ast = {\zeta}^\ast{u}^\ast{\nu}^\ast$ and so every morphism in $\Pi(X)$ has a normal factorization.
\end{proof}
Also let $\pi =\{A_i : i \in I\}$ be a partition of $X$ such that $\pi= \pi_e$ for an idempotent $e\colon X\to X$. For each partition $\pi_\alpha = \{ B_{j} : j \in J\}$, define $e_\alpha \colon  \pi \to {\pi_\alpha} $ by $A_i \mapsto B_j$ such that $(A_i)e \in B_j$. Then a cone $\bar{e}$ defined by 
$$\bar{e}(\bar{\pi}_{\alpha}) = (e_\alpha)^\ast \text{ for all } \bar{\pi}_\alpha \in \Pi(X)$$
is an idempotent normal cone in $\Pi(X)$ with vertex $\bar{\pi}$; and hence $\Pi(X)$ is a normal\index{category!normal} category \cite{pix}. We conclude this section by characterizing the semigroup $T\Pi(X)$ of all normal cones in $\Pi(X)$. 
\begin{thm}
The semigroup $T\Pi(X)$ of all normal cones in $\Pi(X)$ is anti-isomorphic to $Sing(X)$.
\end{thm}
\begin{proof}
Observe that given any transformation $a \in Sing(X)$, we have a normal cone $\sigma^a$ in $\Pi(X)$ as follows. Define $$\sigma^a(\bar{\pi}_\alpha) = (a\alpha)^\ast\text{ for all } \bar{\pi}_\alpha \in v\Pi(X).$$ 
Then $a\alpha_{|\text{Im}a} \colon \text{ Im }a \to \text{ Im }\alpha$ will give a map $\eta_{a\alpha}\colon [x]_{\pi_a} \mapsto [xa\alpha]_{\pi_\alpha}$ from $\pi_a \to \pi_\alpha$. $(a\alpha)^\ast$ will be a morphism from $\bar{\pi}_\alpha$ to $\bar{\pi}_a$ for all $\bar{\pi}_\alpha \in v\Pi(X)$.

So $\sigma^a$ is a normal cone in $\Pi(X)$ with vertex $\bar{\pi}_a$. Also $\sigma^a.\sigma^b = \sigma^{ba}$. Hence $\sigma^a \mapsto a$ is an anti-isomorphism and so the semigroup $T\Pi(X)$ of all normal cones in $\Pi(X)$ is anti-isomorphic\index{isomorphism!of semigroups} to $Sing(X)$.
\end{proof}

\section{Cross-connections}\label{seccxn}
Having characterized the normal categories in $Sing(X)$ as $\mathscr{P}(X)$ and $\Pi(X)$, now we proceed to discuss the cross-connections between them. But for that, first we need to describe the normal duals of $\mathscr{P}(X)$ and $\Pi(X)$. We show below that $\mathscr{P}(X)$ and $\Pi(X)$ are mutual normal duals of each other.
\subsection{Normal duals}
To characterize the normal dual of $\mathscr{P}(X)$, we define functors $P \colon  N^\ast \mathscr{P}(X) \to \Pi(X) $ and $Q \colon  \Pi(X) \to N^\ast \mathscr{P}(X)$ as follows. For $ H(e;-) \in vN^\ast \mathscr{P}(X)$ and $\sigma \colon  H(e;-) \to H(f;-)$,
\begin{equation} \label{eqnP}
P (H(e;-)) = \bar{\pi}_e \quad\text{and}  \quad
P(\sigma) = w^\ast
\end{equation}
where $\pi_e$ is the partition of $X$ induced by the transformation $e$, and $w^\ast$ is given as follows. By Lemma \ref{lemH2}, $\sigma$ induces a unique map $w \colon  \text{ Im }f \to \text{ Im }e$. Then $w^\ast \colon  \bar{\pi}_{e} \to \bar{\pi}_{f} $ is determined by the map $\eta_w\colon  \pi_f \to \pi_e$ induced by $w$. Further $Q$ is defined by
\begin{equation} \label{eqnQ}
Q (\bar{\pi}) = H(e;-) \quad\text{and}  \quad
Q(\eta^\ast) = \sigma
\end{equation}
where $e$ is an idempotent mapping such that $\pi_e = \pi$. Observe that that if $\pi_e=\pi_f$, then $e\mathscr{R}f$ and so $H(e;-)=H(f;-)$ \cite{cross}. For ${\eta}\colon \pi_f \to \pi_e$, $\sigma\colon  H(e;-) \to H({f} ;-)$ is defined by $ a \mapsto {wa} $ where $w\colon  \text{ Im }f \to \text{ Im }e$ is determined uniquely by $\eta$.

$PQ$ is a functor from $N^\ast \mathscr{P}( X)$ to itself such that $PQ = 1_{N^\ast \mathscr{P}( X)}$. Similarly $QP$ is a functor from $\Pi(X)$ to itself such that $QP = 1_{\Pi(X)}$.
Since $P$ is an order isomorphism preserving the inclusions, $N^\ast \mathscr{P}(X)$ is isomorphic to $\Pi(X)$ as normal categories.
\begin{thm} \cite{pix}\label{thmP}
The normal dual $N^\ast \mathscr{P}(  X )$ is isomorphic to the category of partitions $\Pi(X)$. \end{thm}
By duality\index{dual}, we can characterize the normal dual $N^\ast\Pi(X)$ of the partition category. Each normal cone in $\Pi(X)$ can be realized using a transformation $a \in Sing(X)$ and the $H$-functor $H(e;-) \in vN^\ast\Pi(X)$ can be seen as completely characterized by Im $e$. Corresponding to $\sigma\colon H(e;-) \to H(f;-)$ in $N^\ast\Pi(X)$, we have a unique morphism $\eta^\ast\colon \bar{\pi}_{f} \to \bar{\pi}_{e}$ in $\Pi(X)$ and $\eta^\ast$ will determine a map $\eta = \eta_w \colon  \pi_e \to \pi_f$; which in turn gives a unique map $w\colon  \text{ Im }e \to \text{ Im } f$. Hence we can define a functor $R\colon  N^\ast \Pi(X) \to \mathscr{P}(X)$ as follows.
\begin{equation} \label{eqnR}
R (H(e;-)) = \text{ Im }e \quad\text{ and }  \quad
R(\sigma) = w
\end{equation}
We can show that $R$ is an inclusion preserving covariant functor which is an order isomorphism, $v$-injective, $v$-surjective and fully faithful. Hence $N^\ast \Pi(X)$  is isomorphic to the powerset category $\mathscr{P}(X)$. And hence we have the following theorem.
\begin{thm} \label{thmP2}
The normal dual $N^\ast \Pi(X)$ is isomorphic to the powerset category $\mathscr{P}(X)$.
\end{thm}
In the next lemma, we identify the $M$-set\index{$M$-set} associated with a cone $\sigma$ in $\mathscr{P}(X)$, with a cross-section of a partition.
\begin{lem}\label{lemm}
Let $\sigma$ be a normal cone in $\mathscr{P}(X)$. Then 
$$M\sigma \:=  \{ A\subseteq X : A \text{ is a cross-section of } \pi_a \}$$ where $\sigma = \rho^a$ for some $a \in Sing(X)$ and $\pi_a$ is the partition of $a$.
\end{lem}
\begin{proof}
Let $\sigma = \rho^a$ be a normal cone in $\mathscr{P}(X)$. For a regular semigroup $S$, $M{\rho^a} = \{ Se : e \in E(R_a) \}$ where $R_a$ is the $\mathscr{R}$-class of $S$ containing $a$ \cite{cross}.\\
So $M\sigma \:= \{ S b \: : \:  b \in E(R_a) \}$. Hence 
$$M\sigma \:= \{ S b \: : \: \pi_ b = \pi_a \text{ and } b^2 =  b \} \:=\: \{ A\: :\: A \text{ is a cross-section of }\pi_a \}.$$ 
\end{proof}
Thus, given $\sigma \in E(T\mathscr{P}(X))$, we have a natural identification of the $H$-functor $H(\sigma;-)$ with a partition $\pi$; and henceforth we write $M(\pi)$ for $MH(\sigma;-) = M\sigma$.\vspace*{.5cm}

Now we are ready with all the necessary ingredients to discuss the cross-connections between $\mathscr{P}(X)$ and $\Pi(X)$. Since $N^\ast\mathscr{P}(X)$ is isomorphic to $\Pi(X)$ (by Theorem \ref{thmP}), the local isomorphism $\Gamma\colon  \Pi(X) \to N^\ast\mathscr{P}(X)$ of a cross-connection from $\Pi(X)$ to $\mathscr{P}(X)$ can be taken as a local isomorphism $\Gamma\colon  \Pi(X) \to \Pi(X)$. Using Lemma \ref{lemm}, we have the following.
\begin{lem}\label{lemcxn}
$(\Pi(X),\mathscr{P}(X);\Gamma)$ is a cross-connection if $\Gamma\colon  \Pi(X) \to \Pi(X)$ is a local isomorphism such that for every $A \in v\mathscr{P}(X)$, there is some $\bar{\pi} \in v\Pi(X)$ such that $A $ is a cross-section of $\Gamma(\bar{\pi})$.
\end{lem} 
Now given a cross-connection $\Gamma \colon  \Pi(X) \to \Pi(X)\cong N^\ast\mathscr{P}(X) $, we have a bifunctor\index{functor!bifunctor} $\Gamma \colon  \mathscr{P}(X) \times \Pi(X) \to \bf{Set}$ as follows. For all $(A,\bar{\pi}) \in v\mathscr{P}(X) \times v\Pi(X)$ and $(f,v^\ast)\colon (A,\bar{\pi}) \to (B,\bar{\pi}')$ 
\begin{subequations}
\begin{align}
\Gamma (A,\bar{\pi})\:&=\: \{ a \in Sing(X) : \text{ Im } a \subseteq A \text{ and } \bar{\pi}_{a} \subseteq \Gamma(\bar{\pi}) \}\label{eqngx1}\\
\Gamma (f,v^\ast)\:& : \: a \mapsto (ua)f = u(a f)\label{eqngx2}
\end{align}
\end{subequations}
where $u^\ast = \Gamma(v^\ast)$.

By Remark \ref{rmkdual}, given a cross-connection $(\Pi(X),\mathscr{P}(X);\Gamma)$, there exists a unique dual cross-connection $\Delta \colon  \mathscr{P}(X) \to \mathscr{P}(X)$. Just as in the case of $\Gamma$, we have a bifunctor $\Delta \colon  \mathscr{P}(X) \times \Pi(X) \to \bf{Set}$ as follows. For all $(A,\bar{\pi}) \in v\mathscr{P}(X) \times v\Pi(X)$ and $(f,v^\ast)\colon (A,\bar{\pi}) \to (B,\bar{\pi}')$
\begin{subequations}
\begin{align}
\Delta (A,\bar{\pi})\:&=\: \{ a \in Sing(X) : \text{ Im } a \subseteq \Delta(A) \text{ and } \bar{\pi}_{a} \subseteq \bar{\pi} \}\label{eqndx1}\\
\Delta (f,v^\ast)\:& : \: a \mapsto (va)g = v(a g)\label{eqndx2}
\end{align}
\end{subequations}
where $g = \Delta(f)$.

\subsection{Cross-connections induced by permutations on $X$}\label{secper}
Recall that a \emph{permutation}\index{permutation} on a set $X$ is an injective map from $X$ onto itself. In this section, we show that given a permutation $\theta$ on $X$, we can define a cross-connection from the powerset category $\mathscr{P}(X)$ to the partition category $\Pi(X)$ induced by $\theta$. In the sequel, since $\theta$ is often used to represent a functor, we write the argument on the right side of $\theta$ for uniformity; but the order of composition is from left to right as earlier. First we see that a permutation $\theta$ induces an isomorphism $\theta_{\mathscr{P}}$ on the powerset category $\mathscr{P}(X)$. 
\begin{pro}\label{prorps}
Let $\theta$ be a permutation on the set $X$. Define a functor $\theta_{\mathscr{P}} \colon  \mathscr{P}(X) \to \mathscr{P}(X)$ as follows. For a subset $A$ of $X$, define $\theta_{\mathscr{P}}\colon A \mapsto \theta(A)$ and for $f\colon A\to B$ in $\mathscr{P}(X)$, define $ \theta_{\mathscr{P}}(f) = \theta(f)= \theta^{-1} f \theta \colon  \theta(A) \to \theta(B)$. Then $\theta_{\mathscr{P}}$ is a normal category isomorphism\index{isomorphism!of normal categories} on $\mathscr{P}(X)$.
\end{pro}
\begin{proof}
Clearly for $A \in \mathscr{P}(X)$, $\theta_{\mathscr{P}}(A) = \theta(A) \in \mathscr{P}(X)$. Since $\theta$ is injective, $\theta_{\mathscr{P}}$ is injective. $\theta_{\mathscr{P}}$ is onto, since $\theta$ is onto. Hence $\theta_{\mathscr{P}} \colon v\mathscr{P}(X) \to v\mathscr{P}(X)$ is a bijection on $v\mathscr{P}(X)$. Now given $f\colon A\to B$ in $\mathscr{P}(X)$, $ \theta^{-1} f \theta$ is a map $\theta(f)$ from $\theta(A)$ to $\theta(B)$.

Also if $f\colon A \to B$ and $g\colon B \to C$, then $ \theta(f\circ g) = \theta^{-1} (f\circ g) \theta = \theta^{-1} f (\theta \theta^{-1}) g \theta = (\theta^{-1} f \theta)(\theta^{-1} g \theta) = \theta(f)\circ\theta(g)$. Hence $\theta_{\mathscr{P}}$ is a covariant functor on $\mathscr{P}(X)$.

Also given $g \colon  \theta(A) \to \theta(B)$, we see that $\theta g\theta^{-1}$ is a map from $A \to B$; and $\theta(\theta g\theta^{-1}) = \theta^{-1}\theta g\theta^{-1}\theta = g$; and hence $\theta_{\mathscr{P}}$ is full. Also if $\theta^{-1} f \theta = \theta^{-1} g \theta$, then $f=g$ and so $\theta_{\mathscr{P}}$ is faithful.

Now $A\subseteq B$ if and only if $\theta(A) \subseteq \theta(B)$. Also $\theta(j(A,B)) = j(\theta(A),\theta(B))$. Thus $\theta_{\mathscr{P}}$ is a normal category isomorphism on $\mathscr{P}(X)$. 
\end{proof}
When there is no ambiguity regarding notations, we shall denote $\theta_{\mathscr{P}}$ by $\theta$ itself, as above. Now we see that a permutation $\theta$ induces an isomorphism $\theta$ on the partition category $\Pi(X)$ of $X$. For that we need the following lemma which describes a permutation $\theta$ on the set of partitions of $X$.
\begin{lem}\label{lemrp1}
Let $\theta$ be a permutation on the set $X$ and $\sigma = \{ A_i : i \in I \} $ be a partition of $X$. Then $\theta^{-1}(\sigma)$ defined as follows is also a partition of $X$.
\begin{equation}\label{eqnrpt}
\theta ^{-1}(\sigma)\: = \: \{\theta^{-1}(A_i) : i \in I\}
\end{equation}
\end{lem}
Now using $\theta$, we can get an isomorphism of the category of partitions $\Pi(X)$.
\begin{pro}\label{prorpis}
Let $\theta$ be a permutation on the set $X$. Define a functor $\theta \colon  \Pi(X) \to \Pi(X)$ as follows. For $\bar{\sigma} \in v\Pi(X)$, $ \theta(\bar{\sigma}) = {\theta}^\ast(\bar{\sigma}) = \overline{\theta^{-1}(\sigma)}$ and for $\eta^\ast\colon \bar{\sigma}_{1}\to \bar{\sigma}_{2}$ in $\Pi(X)$, $\theta(\eta^\ast) = (\theta \eta \theta^{-1})^\ast$. Then this is a normal category isomorphism\index{isomorphism!of normal categories} on $\Pi(X)$.
\end{pro}
\begin{proof}
Since $\theta$ is a mapping from $\theta ^{-1}(\sigma)$ to $\sigma$, $\theta^\ast $ is a morphism in $\Pi(X)$ from $\overline{\sigma} $ to $ \overline{\theta^{-1}(\sigma)}$ such that $\theta^\ast\colon  \alpha \mapsto \theta\alpha$ for all $\alpha \in \bar{\sigma}$.
Since $\theta$ is bijective, $\theta$ (as a functor) is a bijection on $v\Pi(X)$ such that $\overline{\sigma} \mapsto \overline{\theta^{-1}(\sigma)}$. Now $\theta$ being a bijection can be inverted and hence given $\eta^\ast\colon \bar{\sigma}_{1}\to \bar{\sigma}_{2}$ in $\Pi(X)$, we see that $(\theta^{-1})^* \eta^\ast \theta^* = (\theta \eta \theta^{-1})^\ast$ gives a map from $\theta^*(\bar{\sigma}_{1})$ to $\theta^*(\bar{\sigma}_{2})$ with the corresponding mapping $\theta(\eta) = \theta\eta \theta^{-1}$ from the partition $\theta^{-1}(\sigma_2)$ to $\theta^{-1}(\sigma_1)$.

If $\eta^\ast\colon \bar{\sigma}_{1} \to \bar{\sigma}_{2}$ and $\mu^\ast\colon \bar{\sigma}_{2} \to \bar{\sigma}_{3}$, then $ \theta(\eta^\ast\circ \mu^\ast) = (\theta^{-1})^* (\eta^\ast\circ \mu^\ast) \theta^* = (\theta^{-1})^* \eta^\ast (\theta^* (\theta^{-1})^*) \mu^\ast \theta^* = ((\theta^{-1})^* \eta^\ast \theta^*)((\theta^{-1})^* \mu^\ast \theta^*) = \theta(\eta^\ast)\circ\theta(\mu^\ast)$. Hence $\theta$ is a covariant functor on $\Pi(X)$

Given $\mu^\ast \colon  \theta(\bar{\sigma}_{1}) \to \theta(\bar{\sigma}_{2})$, then $\theta^* \mu^\ast (\theta^{-1})^* $ is a map from $\bar{\sigma}_{1} \to \bar{\sigma}_{2}$; and hence $\theta$ is full. If $(\theta^{-1})^* \eta^\ast \theta^* = (\theta^{-1})^* \mu^\ast \theta^*$, then $\eta^\ast=\mu^\ast$ and so $\theta$ is faithful.

Now $\bar{\sigma}_{1}\subseteq \bar{\sigma}_{2}$ if and only if ${\sigma_1} \supseteq {\sigma_2}$ if and only if $\theta^{-1}(\sigma_1) \supseteq \theta^{-1}(\sigma_2)$ if and only if $\theta^*(\bar{\sigma}_{1}) \subseteq \theta^*(\bar{\sigma}_{2})$. Also $\theta(j(\bar{\sigma}_{1},\bar{\sigma}_{2})) = j(\theta^*(\bar{\sigma}_{1}),\theta^*(\bar{\sigma}_{2}))$. Thus $\theta$ is a normal category isomorphism on $\Pi(X)$. 
\end{proof}
\begin{lem}\label{crossrp}
Let $\theta$ be a permutation of $X$. Then $(\Pi(X),\mathscr{P}(X); \Gamma_\theta)$ is a cross-connection\index{cross-connection} where $\Gamma_\theta \colon  \Pi(X) \to \Pi(X) \text{ is defined by }$ 
\begin{equation}\label{eqncrossrp}
\Gamma_\theta (\bar{\pi}) \:=\: \theta(\bar{\pi}) \text{ and } \Gamma_\theta(\eta^\ast) = \theta(\eta^\ast)
\end{equation}
\end{lem}
\begin{proof}
By Proposition \ref{prorpis}, $\Gamma_\theta$ is a normal category isomorphism of $\Pi(X)$, $\Gamma_\theta$ is inclusion preserving, fully faithful and for each $\bar{\pi} \in v\Pi(X)$, $F_{|\langle \bar{\pi} \rangle}$ is an isomorphism of the ideal $\langle \bar{\pi} \rangle$ onto $\langle \Gamma_\theta(\bar{\pi}) \rangle$. Also since the image of $\Gamma_\theta$ is full, for every $A \in v\mathscr{P}(X)$, there is some $\bar{\pi} \in v\Pi(X)$ such that $A \in M{\theta(\bar{\pi})}$. Hence from Lemma \ref{lemcxn}, $\Gamma_\theta$ is a cross-connection.
\end{proof}
As above, it is straightforward to show the following.
\begin{lem}
Let $\theta$ be a permutation of $X$ and $\Delta_{\theta} \colon  \mathscr{P}(X) \to \mathscr{P}(X)$ be a functor such that 
\begin{equation}\label{eqndcrossrp}
\Delta_{\theta} (A) \:=\: \theta(A) \text{ and } \Delta_{\theta}(f) = \theta(f)
\end{equation}
Then $(\mathscr{P}(X), \Pi(X) ; \Delta_{\theta})$ is a cross-connection\index{cross-connection!dual}.
\end{lem}
\begin{pro}\label{prorcross}
Let $(\Pi(X),\mathscr{P}(X); \Gamma_\theta)$ and $(\mathscr{P}(X), \Pi(X); \Delta_{\theta})$ be two cross-connections with associated bifunctors $\Gamma_{\theta }(-,-)$ and $\Delta_{\theta }(-,-)$ respectively, then there is a natural isomorphism $\chi_{\Gamma_\theta} \colon  \Gamma_{\theta }(-,-) \to \Delta_{\theta }(-,-)$ defined as 
$$\chi_{\Gamma_\theta}(A,\bar{\pi})\colon  a \mapsto \theta^{-1} a\theta.$$ 
Hence $\Delta_{\theta }$ is the dual of ${\Gamma_\theta}$ and $\chi_{\Gamma_\theta}$ is the duality associated with the cross-connection.
\end{pro}
\begin{proof}
First, let $(f,\eta^\ast) \in \mathscr{P}(X) \times \Pi(X)$ where $f \colon  A \to B$ and $\eta^\ast \colon  \bar{\pi}_1 \to \bar{\pi}_2$. 
\begin{equation*}\label{duality}
\xymatrixcolsep{4pc}\xymatrixrowsep{3pc}\xymatrix
{
 \Gamma_{\theta }(A,\bar{\pi}_1) \ar[r]^{\chi_{\Gamma_\theta}(A,\bar{\pi}_1)} \ar[d]_{\Gamma_{\theta }(f,\eta^\ast)}  
 & \Delta_{\theta }(A,\bar{\pi}_1) \ar[d]^{\Delta_{\theta }(f,\eta^\ast)} \\       
 \Gamma_{\theta }(B,\bar{\pi}_2) \ar[r]^{\chi_{\Gamma_\theta}(B,\bar{\pi}_2)} & \Delta_{\theta }(B,\bar{\pi}_2) 
}
\end{equation*}
To see that the above diagram commutes, let $ a \in \Gamma_{\theta }(A,\bar{\pi}_1)$. Then,
\begin{align*}
( a)\chi_{\Gamma_\theta}(A,\bar{\pi}_1)\Delta_{\theta }(f,\eta^\ast)& = (\theta^{-1}  a \theta ) \Delta_{\theta }(f,\eta^\ast)\\
& =  \eta (\theta^{-1}  a \theta) \theta(f) \text{ (using (\ref{eqndx2}))}\\
& = \eta \theta^{-1}  a \theta \theta^{-1} f \theta  \text{ (using Proposition \ref{prorps})}\\
& = \eta \theta^{-1}  a f \theta. 
\end{align*}
\begin{align*}
\text{Also } ( a)\Gamma_{\theta }(f,\eta^\ast)\chi_{\Gamma_\theta}(B,\bar{\pi}_2) &= (\theta(\eta) a f) \chi_{\Gamma_\theta}(B,\bar{\pi}_2) \text{ (using (\ref{eqngx2}))}\\
 &= \theta^{-1}(\theta(\eta) a f)\theta \\
 &= \theta^{-1}\theta\eta\theta^{-1} a f\theta \text{ (using Proposition \ref{prorpis})}\\
 &= \eta\theta^{-1} a f\theta.
\end{align*}
Thus we have 
$$( a)\chi_{\Gamma_\theta}(A,\bar{\pi}_1)\Delta_{\theta }(f,\eta^\ast) = ( a)\Gamma_{\theta }(f,\eta^\ast)\chi_{\Gamma_\theta}(B,\bar{\pi}_2) \:\: \text{ for every } \:  a \in \Gamma_{\theta }(A,\bar{\pi}_1).$$ 
Now for $ a \in \Gamma_{\theta }(A,\bar{\pi})$, the map $\theta^{-1}  a \theta \in \Delta_{\theta }(A,\bar{\pi})$. To see that,
observe that Im $ a  \subseteq A$, Im $\theta^{-1}  a \subseteq A$ imply Im $\theta^{-1}  a \theta \subseteq \theta(A) $.

Similarly since $\bar{\pi}_{ a} \subseteq \theta(\bar{\pi})$ and $\theta(\bar{\pi}) ={\theta}^\ast(\bar{\pi})$, we have $\theta(\pi) \subseteq \pi_ a$. Since $\pi_ a \subseteq \pi_{ a\theta}$, we have $\theta(\pi) \subseteq \pi_{ a\theta}$, i.e., $\theta^{-1}\theta(\pi) \subseteq \theta^{-1}(\pi_{ a\theta})$. Now since $\theta^{-1}\theta(\pi) = \pi$ and $\theta^{-1}(\pi_{ a\theta}) = \pi_{\theta^{-1} a\theta}$, we have $\pi \subseteq \pi_{\theta^{-1} a\theta}$. Hence $\overline{\pi_{\theta^{-1} a\theta}} \subseteq \overline{\pi}$. Thus the map $\theta^{-1}  a \theta \in \Delta_{\theta }(A,\bar{\pi})$ whenever $ a \in \Gamma_{\theta }(A,\bar{\pi})$.

Also observe that since the conjugation map $ a \mapsto \theta^{-1} a\theta$ is a bijection, the map $\chi_{\Gamma_\theta}(A,\bar{\pi})$ is a bijection of the set $\Gamma_{\theta }(A,\bar{\pi})$ onto the set $\Delta_{\theta }(A,\bar{\pi})$. Thus $\chi_{\Gamma_\theta}$ is a natural isomorphism. 
\end{proof}
The natural isomorphism $\chi_{\Gamma_\theta}$ is the duality associated with $\Gamma_\theta$ and by \cite{cross}, $\Delta_{\theta }$ is the dual cross-connection of ${\Gamma_\theta}$.

Given a cross-connection $(\Pi(X),\mathscr{P}(X); \Gamma_\theta)$ with a dual cross-connection $(\mathscr{P}(X), \Pi(X) ; \Delta_{\theta})$, from the general theory, there exists two regular subsemigroups $U\Gamma_\theta$ and $U\Delta_{\theta}$ of $T\mathscr{P}(X)$ and $T\Pi(X)$ respectively. We will identify the cross-connection semigroup as a sub-direct product of these intermediary subsemigroups. $U\Gamma_\theta$ and $U\Delta_{\theta}$ are described in the following lemma.
\begin{lem}\label{lemrtugamma}
$$U\Gamma_\theta \:=\: \bigcup\{  a \in Sing(X) :\text{ Im }  a \subseteq A \text{ and } \bar{\pi}_{ a} \subseteq \theta(\bar{\pi}) \: : \: (A,\bar{\pi}) \in v\mathscr{P}(X) \times v\Pi(X)\}$$
$$U\Delta_{\theta}   \:=\: \bigcup\{  b \in Sing(X) :\text{ Im }  b \subseteq \theta(A) \text{ and } \bar{\pi}_{ b} \subseteq \bar{\pi} \: : \: (A,\bar{\pi}) \in v\mathscr{P}(X) \times v\Pi(X)\}$$
\end{lem}
Observe here that since $(A,\bar{\pi}) \in v\mathscr{P}(X) \times v\Pi(X)$, $U\Gamma_\theta \cong Sing(X)$ and $U\Delta_{\theta} \cong Sing(X)^\text{op}$. The following lemma which describes the linked cones can be easily verified. 
\begin{lem}\label{lemrtlinked}
Let $(\Pi(X),\mathscr{P}(X); \Gamma_\theta)$ be a cross-connection such that the duality is given by $\chi_{\Gamma_\theta}(A,\bar{\pi}) \colon   a \mapsto \theta^{-1} a\theta$; then $ a$ is linked to $ b$ if and only if $  b = \theta^{-1} a\theta$.
\end{lem}
Having characterized the linked cones for the cross-connection, now we describe the cross-connection semigroup\index{cross-connection!semigroup}.
\begin{thm}\label{thmcrsrtx}
Given a cross-connection $(\Pi(X),\mathscr{P}(X); \Gamma_\theta)$, the cross-connection semigroup\index{isomorphism!of semigroups} $\tilde{S}\Gamma_\theta$ is isomorphic to $Sing(X)$.
\end{thm}
\begin{proof}
Using Lemma \ref{lemrtlinked}, we see that 
$$ \tilde{S}\Gamma_\theta = \:\{\: ( a,\theta^{-1} a\theta) \: \text{ such that }  a \in Sing(X)\}$$
Since $U\Delta_{\theta}$ is a subsemigroup of $Sing(X)^{\text{op}}$, with operation $\theta^{-1} a'\theta\ast\theta^{-1} a\theta = \theta^{-1} a\theta\theta^{-1} a'\theta = \theta^{-1} a(\theta\theta^{-1}) a'\theta = \theta^{-1} a a'\theta$ in $Sing(X)^{\text{op}}$, so the binary operation in $\tilde{S}\Gamma_\theta$ is given by $$( a,\theta^{-1} a\theta) \circ ( a',\theta^{-1} a'\theta) = ( a a',\theta^{-1} a a'\theta)$$
Now define a mapping $\psi\colon  Sing(X) \to \tilde{S}\Gamma_\theta $ as $ a\mapsto ( a,\theta^{-1} a\theta)$. Clearly $\psi$ is well-defined and onto. Also $( a \cdot  a')\psi = ( a a')\psi = ( a a',\theta^{-1}  a a'\theta) = ( a a',\theta^{-1}  a\theta\theta^{-1} a'\theta) = ( a,\theta^{-1} a\theta) \circ ( a',\theta^{-1} a'\theta) = ( a)\psi \circ ( a')\psi$. Thus $\psi$ is a homomorphism. 
Now if $( a)\psi = ( a')\psi$, then $( a,\theta^{-1} a\theta) = ( a',\theta^{-1} a'\theta)$ and hence $ a = a'$. Thus $\psi$ is an isomorphism and hence $\tilde{S}\Gamma_\theta$ is isomorphic to $Sing(X)$.
\end{proof}
\begin{rmk}
Thus given a permutation $\theta$ on $X$, we can get a representation of $Sing(X)$ as the cross-connection semigroup $\tilde{S}\Gamma_\theta$ corresponding to the cross-connection $(\Pi(X),\mathscr{P}(X); \Gamma_\theta)$.

Let $Sing(X)^* = (Sing(X),\ast)$ with the binary composition $\ast$ defined as follows. For a permutation $\theta$,
$$  a \ast  b =  a \cdot \theta \cdot  b \quad \text{ for }  a,  b \in Sing(X).$$
Then we can see that the representation $\Gamma_\theta$ described above refers to the cross-connection arising from the semigroup $Sing(X)^*$.

If we define $\phi\colon  Sing(X)^* \to \tilde{S}\Gamma_\theta$ as $ a \mapsto (\theta a, a\theta)$, then 
$$( a \ast  a')\phi = ( a\theta a')\phi = (\theta a\theta a', a\theta a'\theta ) = (\theta a,  a\theta)\circ(\theta  a', a'\theta) = ( a)\phi \circ ( a')\phi.$$ 
Further it can be seen that $\phi$ is an isomorphism. 
\end{rmk}

Now we proceed to show that in fact these are all the cross-connections from ${\Pi(X)}$ to ${\mathscr{P}(X)}$. 
\begin{thm}\label{thmmain}
Any cross-connection $\Gamma$ from ${\Pi(X)}$ to ${\mathscr{P}(X)}$ is of the form $\Gamma_\theta$ for some permutation $\theta$.
\end{thm}
\begin{proof}
Recall that a cross-connection $\Gamma\colon {\Pi(X)} \to {\mathscr{P}(X)}$ is a local isomorphism $\Gamma\colon  {\Pi(X)} \to {\Pi(X)}$ such that there is an associated dual cross-connection $\Delta \colon {\mathscr{P}(X)} \to {\mathscr{P}(X)}$ (see Definition \ref{cxn} and Remark \ref{rmkdual}). First observe that since $\Delta$ preserves inclusion and is an order isomorphism on ideals, if $A\subseteq B$ then $\Delta(A) = j(A,B)\Delta(B)$.

Then define 
$$(x)\theta = y \text{ such that } \Delta(\{x\}) = \{y\}.$$
Since $\Delta$ is a local isomorphism, $\Delta_{|\langle x \rangle}$ will be a bijection and hence $\theta$ is well defined.

Further $\theta$ is injective. Suppose that it is not injective. Then there exist $a, a' \in X$ such that $(a)\theta =(a')\theta =b$. This implies $\Delta(\{a\}) = \Delta(\{a'\}) = \{b\}$. Then $\Delta(\{a,a'\}) = \{b\}$. But this is clearly not possible as $\Delta(\{a,a'\})$ is not a bijection. Hence $\theta$ is injective.

Also $\theta$ is onto. Suppose that it is not onto. Then there exists $b\in X$ such that $\Delta(\{x\}) \neq \{b\}$ for any $x\in X$. That is $\Delta(X) \subseteq X\backslash \{b\}$. Now since $\Delta$ is a cross-connection, for every $\bar{\pi} \in {\Pi(X)}$ there exists a subset $A \in \mathscr{P}(X)$ such that $\Delta(A)$ is a cross-section of ${\pi}$. If we choose $\pi =\{X\backslash\{ b\}, \{ b\} \}$, then $\pi$ is a partition of $X$ and any cross-section $C$ of $\pi$ should contain $b$. But since $\Delta(X) \subseteq X\backslash \{b\}$, no such $C$ exists and this leads to a contradiction. Hence $\theta$ is onto.

Since $A\subseteq B$ implies $\Delta(A) = j(A,B)\Delta(B)$, we can argue as in \cite{ltx} (see (\ref{cone}) above) to show that any cross-connection $\Delta$ from ${\mathscr{P}(X)}$ to ${\Pi(X)}$ is one of the form $\Delta_{\theta}$ as defined by (\ref{eqndcrossrp}). Here for any $A \in v{\mathscr{P}(X)}$, $\Delta(A) = \theta(A)$ and given $f\colon  A \to B$, $\theta(f) = \theta^{-1} f \theta$ is the morphism from $\theta(A)$ to $\theta(B)$. Now by the duality of the cross-connections \cite{cross}, it can be shown (see Remark \ref{rmkdual} and Proposition \ref{prorcross}) that the dual cross-connection of $\Delta_{\theta}$ is $\Gamma_\theta$. This implies that $\Gamma = \Gamma_\theta$ and hence the theorem follows. 
\end{proof}
Since the cross-connection semigroup $\tilde{S}\Gamma_\theta$ is isomorphic to $Sing(X)$, we have the following corollary. 
\begin{cor}
Every cross-connection semigroup arising from the cross-connections between $\Pi(X)$ and $\mathscr{P}(X)$ is isomorphic to $Sing(X)$.
\end{cor}

Thus, in this section, we have shown that although different permutations give rise to different cross-connections, the resulting cross-connection semigroups are all isomorphic.

We have also seen that $Sing(X)$, although not a monoid, behaves much like one. The fact that it is constructible from one of $\mathscr{P}(X)$ or $\Pi(X)$ is worth noticing. It is the result of the strong lattice structure present among the principal ideals of $Sing(X)$. In fact, it is exactly this structure which restricts the \emph{cross-connections} between the underlying categories; which in turn results in all cross-connection smeigroups to be isomorphic.  

\section{Right reductive subsemigroups of $Sing(X)$}\label{secrrs}
One class of semigroups we can construct using cross-connections are the right reductive subsemigroups of $Sing(X)$ with principal left ideals isomorphic to $\mathscr{P}(X)$. Recall that a semigroup $S$ is said to be \emph{right reductive}\index{semigroup!right reductive} if the \emph{right-regular representation}\index{representation!right-regular} of $S$, $\rho\colon a\mapsto \rho_a$ is injective \cite{clif}. 

For constructing them we need to identify the total ideals\index{ideal!total} of $\Pi(X)$. An ideal $\mathcal{I}$ in the normal dual $N^\ast\mathcal{C}$ of a normal category $\mathcal{C}$ is said to be a \emph{total ideal}\index{ideal!total} in $N^\ast\mathcal{C}$ if for all $c \in v\mathcal{C}$, there is some $H(\epsilon;-) \in v\mathcal{I}$ such that $c \in MH(\epsilon;-)$. Using Lemma \ref{lemm}, we have the following.
\begin{lem}\label{lemtot}
An ideal $\mathcal{I}$ is a total ideal of $\Pi(X)$ if and only if for every proper maximal subset $A$ of $X$, there exists $\bar{\pi} \in v\mathcal{I}$ such that $A$ is a cross-section\index{cross-section} of $\pi$.
\end{lem}
\begin{proof}
Suppose $\mathcal{I}$ is a total ideal of $\Pi(X)$ and let $A$ be a maximal proper subset of $X$. Then by the identification in  (\ref{eqnP}) of $N^*\mathscr{P}(X)$ with $\Pi(X)$, there is some $\bar{\pi} \in v\mathcal{I}$ such that $A \in MH(e;-)$ where $\pi$ is the partition induced by $e$. By Lemma \ref{lemm}, this implies that for every proper maximal subset $A$ of $X$, there exists $\bar{\pi} \in v\mathcal{I}$ such that $A$ is a cross-section of $\pi$.

Conversely assume that for every proper maximal subset $A$ of $X$, there exists $\bar{\pi} \in v\mathcal{I}$ such that $A$ is a cross-section of $\pi$. Since $\mathcal{I}$ is an ideal, for every subset $A$ of $X$ there exists $\bar{\pi} \in v\mathcal{I}$ such that $A$ is a cross-section of $\pi$ and using Lemma \ref{lemm}, we see that $\mathcal{I}$ is a total ideal of $\Pi(X)$.
\end{proof}
\begin{lem}\label{lemrr1}
If $\mathcal{I}$ is a total ideal of $\Pi(X)$, then $(\mathcal{I},\mathscr{P}(X); \Gamma)$ is a cross-connection\index{cross-connection} where 
\begin{equation}\label{eqngammatx}
\Gamma \colon  \mathcal{I} \to \Pi(X) \text{ is the inclusion functor } j(\mathcal{I},\Pi(X))
\end{equation}
\end{lem}
\begin{proof}
By Lemma \ref{lemcxn}, $(\mathcal{I},\mathscr{P}(X);\Gamma)$ is a cross-connection if $\Gamma\colon  \mathcal{I} \to \Pi(X)$ is a local isomorphism\index{local isomorphism} such that for every $A \in v\mathscr{P}(X)$, there is some $\bar{\pi} \in v\mathcal{I}$ such that $A $ is a cross-section of $\Gamma(\bar{\pi})$. Since $\Gamma$ is an inclusion functor, $\Gamma$ is inclusion preserving and faithful. Since $\mathcal{I}$ is a full subcategory of $\Pi(X)$, $\Gamma$ is full. Also for each ideal $\langle \bar{\pi} \rangle$ for $\bar{\pi} \in v\mathcal{I}$, $\Gamma_{|\langle \bar{\pi} \rangle}$ is the identity morphism from $\langle \bar{\pi} \rangle $ to $ \langle \Gamma(\bar{\pi}) \rangle = \langle \bar{\pi} \rangle $. So $\Gamma$ is a local isomorphism. Since $\mathcal{I}$ is a total ideal, by Lemma \ref{lemtot}, for every subset $A$ of $X$, there exists $\bar{\pi} \in v\mathcal{I}$ such that $A$ is a cross-section of $\pi $. Hence $\Gamma$ is a cross-connection.
\end{proof}
\begin{lem}
Let $\Delta $ be the identity functor\index{functor!identity} from $\mathscr{P}(X)$ to $\mathscr{P}(X)$. Then for a total ideal $\mathcal{I}$ of $\Pi(X)$, the triplet $(\mathscr{P}(X), \mathcal{I} ; \Delta)$ is a cross-connection\index{cross-connection!dual} such that for $A \in v\mathscr{P}(X)$ and $\bar{\pi} \in v\Pi(X)$, $A \in M\Gamma(\bar{\pi}) $ if and only if $\bar{\pi} \in M\Delta(A)$. Hence $\Delta$ is the dual cross-connection of $\Gamma$.
\end{lem}
\begin{proof}
First, since $\mathcal{I}$ is a total ideal of $\Pi(X)$, the category $N^*\mathcal{I}$ is isomorphic to $\mathscr{P}(X)$. Hence $\Delta = 1_{\mathscr{P}(X)}$ is inclusion preserving, fully faithful and for each $A \in v\mathscr{P}(X)$, the restriction $F_{|\langle A \rangle}$ is an isomorphism (in fact identity functor) of the ideal $\langle A \rangle$ onto $\langle \Delta(A) \rangle = \langle A \rangle $. Also clearly the image of $v\Delta$ is total in $\mathscr{P}(X)$ and hence $\Delta$ is a cross-connection.

Then for $A \in v\mathscr{P}(X)$ and $\bar{\pi} \in v\Pi(X)$, let $A \in M\Gamma(\bar{\pi}) $. Since $\Gamma = j(\mathcal{I},\Pi(X))$, we have $\Gamma(\bar{\pi}) = \bar{\pi}$. Then by the dual of Lemma \ref{lemm}, $A$ is a cross-section of $\pi$. Now since $\Delta = 1_{\mathscr{P}(X)}$, we observe that $\Delta(A)= A$ and so $\bar{\pi} \in M\Delta(A)$. Similarly the converse also holds. Hence by \cite{cross}, $\Delta$ is the dual cross-connection of $\Gamma$.
\end{proof}
Observe that although each cross-connection has a unique dual cross-connection, it is possible for different cross-connections to have the same dual, as seen in the above lemma. Now given a cross-connection $\Gamma$ and its dual $\Delta$, the semigroups $U\Gamma$ and $U\Delta$ associated with them is described in the next lemma.
\begin{lem}\label{lemugamma}
Given a cross-connection $\Gamma = j(\mathcal{I},\Pi(X))$ and the dual cross-connection $\Delta = 1_{\mathscr{P}(X)}$,
$$U\Gamma \:=\: \bigcup\{  a \in Sing(X) :  \text{ Im }  a \subseteq A \text{ and } \bar{\pi}_{ a} \subseteq \bar{\pi}  \: : \: (A,\bar{\pi}) \in v\mathscr{P}(X) \times v\mathcal{I}\}$$
$$U\Delta \:=\: \bigcup\{  b \in Sing(X) :  \text{ Im }  b \subseteq A \text{ and } \bar{\pi}_{ b} \subseteq \bar{\pi} \: : \: (A,\bar{\pi}) \in v\mathscr{P}(X) \times v\mathcal{I}\} .$$
\end{lem}
The next lemma characterizes the duality associated with the cross-connection defined in  (\ref{eqngammatx}) and the corresponding linked cones. 
\begin{pro}\label{procross}
Given a cross-connection $\Gamma = j(\mathcal{I},\Pi(X))$ and the dual cross-connection $\Delta = 1_{\mathscr{P}(X)}$; $\chi_{\Gamma} \colon  \Gamma(-,-) \to \Delta(-,-)$ is defined as $\chi_{\Gamma}(A,\bar{\pi}) \colon   a \mapsto  a$. Hence $ a$ is linked to $ b$ if and only if $ a =  b$.
\end{pro}
\begin{proof}
First observe that given the cross-connection $\Gamma \:=\: j(\mathcal{I},\Pi(X))$, $$\Gamma(A,\bar{\pi}) = \Delta(A,\bar{\pi}) = \{  a : \text{ Im }  a \subseteq A \text{ and } \bar{\pi}_{ a} \subseteq \bar{\pi} \}.$$ So if we define $\chi_{\Gamma}(A,\bar{\pi}) \colon   a \mapsto  a$, then $\chi_{\Gamma}(A,\bar{\pi}) = 1_{\Gamma(A,\bar{\pi})}$ and hence is clearly a natural isomorphism.

Now given a cross-connection $\Gamma \:=\: j(\mathcal{I},\Pi(X))$, we shall say that $ a \in U\Gamma$ is linked\index{cone!linked} to $ b \in U\Delta$ if there is a $(A,\bar{\pi}) \in v\mathscr{P}(X) \times v\mathcal{I}$ such that $ a \in \Gamma(A,\bar{\pi})$ and $  b = \chi_\Gamma(A,\bar{\pi})( a)$. Hence it is clear that $ a$ is linked to $ b$ if and only if $ a =  b$.
\end{proof}
Having characterized the linked cones for the cross-connection, now we are in a position to describe the cross-connection semigroup\index{cross-connection!semigroup}.
\begin{thm}\label{thmcrstx}
Given a cross-connection $\Gamma \:=\: j(\mathcal{I},\Pi(X))$, the cross-connection semigroup $\tilde{S}\Gamma$ is isomorphic\index{isomorphism!of semigroups} to a right reductive\index{semigroup!right reductive} regular subsemigroup of $Sing(X)$.
\end{thm}
\begin{proof}
The cross-connection semigroup $\tilde{S}\Gamma$ defined as 
$$ \tilde{S}\Gamma = \{ ( a, b) \in U\Gamma\times U\Delta : ( a, b) \text{ is linked }\} $$ is a regular semigroup with the binary operation defined by $ ( a, b) \circ ( a', b') = ( a a', b'\ast b) $. Using Proposition \ref{procross}, we see that $ \tilde{S}\Gamma = \{ ( a, a) :  a \in U\Gamma\}$. Now since the multiplication in the semigroup $U\Delta$ is defined as $ a'\ast  a  =  a .  a' \text{ for all }  a, a' \in U\Delta$; the binary operation in $ \tilde{S}\Gamma$ is $( a, a) \circ ( a', a') = ( a a', a a')$.

Now define a mapping $\phi\colon  \tilde{S}\Gamma \to Sing(X)$ as $( a, a)\mapsto  a$. Clearly $\phi$ is well-defined and injective. Also 
$$(( a, a) \circ ( a', a'))\phi = ( a a', a a')\phi =  a a' = ( a, a)\phi \circ ( a', a')\phi.$$ Thus $\phi$ is a homomorphism. Using the regularity of $\tilde{S}\Gamma$, we see that $\tilde{S}\Gamma$ is isomorphic to a regular subsemigroup of $Sing(X)$.

Now the projection $p\colon  \tilde{S}\Gamma \to U\Gamma$ given by $( a, a) \mapsto  a$ is clearly an isomorphism. Then $p= \rho \psi$ where $\rho$ is the right regular representation of $\tilde{S}\Gamma$ and $\psi\colon (\tilde{S}\Gamma)_\rho\to U\Gamma$ is an isomorphism \cite{cross}. Hence $\rho$ will be an isomorphism and so $\tilde{S}\Gamma$ is right reductive. Hence the theorem.
\end{proof}
\begin{cor}
In particular if we take $\mathcal{I} = \Pi(X)$ and the identity functor on $\Pi(X)$ as the cross-connection $\Gamma$, then the cross-connection semigroup $\tilde{S}\Gamma $ of $(\Pi(X),\mathscr{P}(X);\Gamma)$ is isomorphic\index{isomorphism!of semigroups} to $Sing(X)$.
\end{cor} 
In fact, all the right reductive subsemigroups of $Sing(X)$ with the principal left ideal category isomorphic to the powerset category $\mathscr{P}(X)$ is one of the form $\tilde{S}\Gamma$ \cite{cross}. We conclude by illustrating an example. 
\begin{ex}
If $|X| = n $; then there are $(^n_2)$ minimal partitions since any minimal partition will have only one set of two elements and all other sets singleton. Also there are $n$ maximal subsets. For each minimal partition $\pi$, there are exactly two cross-sections and hence two idempotent mappings with the given partition. For each maximal subset $A$, there are $n-1$ idempotents with image $A$. Choose at least $m=(^n_2)-(n-2)$ minimal partitions and let $\mathcal{I} \:=\:  \cup_{i=1}^{m}\{\bar{\pi}_{i} \: \cup \: \{ \bar{\pi}_{j} : \bar{\pi}_{j} \subseteq \bar{\pi}_{i}\}\}$ and do the construction as in Section \ref{secrrs}, we get the cross-connection semigroup $\tilde{S}\Gamma = T$ as a right reductive subsemigroup of $Sing(X)$. Here $\tilde{S}\Gamma $ corresponds to the subsemigroup $T$ of $Sing(X)$ given by $T= \{  a \in Sing(X) \: : \: \bar{\pi}_{ a} \in v\mathcal{I} \}$.

For example, choose $n=5$ and $X=\{ 1,2,3,4,5\}$. Then $m= (^5_2)-(5-2) = 7$. Now define $\mathcal{I}$ by choosing all minimal partitions except $\pi_1 =\{(1 2)(3)(4)(5)\}$. Then $T = \{  a \in Sing(X) : \pi_ a \neq \pi_1\}$, i.e., it contains all elements of $Sing(X)$ except 120 elements of the form $(aabcd)$ where $a,b,c,d \in X$. 
\end{ex}
\section*{Acknowledgements}
\noindent The authors express their deep gratitude to M. V. Volkov, Ural Federal University, Russia for his helpful advice, suggestions and remarks.\\
We are very grateful to the referee for the careful reading of the paper
and for his/her comments and detailed suggestions which helped us to
improve the manuscript considerably.\\
The first author acknowledges the financial support of IISER, Thiruvananthapuram in the preparation of this paper. On the final stage, his work was also supported by the Competitiveness Enhancement Program of Ural Federal University, Russia.\\
The second author acknowledges the financial support of the KSCSTE, Thiruvananthapuram (via the award of Emeritus scientist) in the preparation of this paper.\\
\bibliographystyle{unsrt}
\bibliography{aa}

\end{document}